\documentclass{amsart}
\usepackage[a4paper,left=3cm,right=3cm, top=3cm, bottom=3cm]{geometry}
\usepackage[latin1]{inputenc}
\usepackage{enumerate}
\usepackage{mathrsfs}
\usepackage{comment}
\usepackage[T1]{fontenc}
\usepackage {graphicx}
\usepackage{amsmath}
\usepackage{amsthm}
\usepackage{amssymb}
\usepackage{hyperref}
\theoremstyle{plain}
\newtheorem{thm}{Theorem}[section]  
\newtheorem{cor}[thm]{Corollary}
\newtheorem{lem}[thm]{Lemma}
\newtheorem{prop}[thm]{Proposition}
\theoremstyle{definition}
\newtheorem{defn}[thm]{Definition}

\theoremstyle{remark}

\newcommand{\R}{\mathbb{R}}

\newcommand{\U}{\mathcal{U}}
\newcommand{\I}{\mathcal{I}}
\newcommand{\X}{\mathcal{X}}
\newcommand{\Y}{\mathcal{Y}}
\newcommand{\F}{\mathcal{F}}

\DeclareMathOperator{\supp}{supp}
\DeclareMathOperator{\suc}{\uparrow}

\DeclareMathOperator{\Lev}{Lev}

\DeclareMathOperator{\htte}{ht}

\DeclareMathOperator{\Fin}{Fin}

\newcommand{\vertiii}[1]{{\left\vert\kern-0.25ex\left\vert\kern-0.25ex\left\vert #1 
    \right\vert\kern-0.25ex\right\vert\kern-0.25ex\right\vert}}

\begin{document}
\title{On bisequentiality and spaces of strictly decreasing functions on trees}
\author{Claudio Agostini}
\author{Jacopo Somaglia}
\thanks{The research of the second author was supported in part by the 
Universit\`a degli Studi of Milano (Italy), in part by the Gruppo Nazionale per l'Analisi Matematica, la Probabilit\`a e le loro Applicazioni (GNAMPA) of the Istituto Nazionale di Alta Matematica (INdAM) of Italy and in part by the research grant GA\v{C}R 17-00941S}
\address{Universit\`a degli Studi di Milano, Dipartimento di Matematica ``F. Enriques'', Via Cesare Saldini 50, 20133 Milano, Italy.}

\email{claudio.agostini@studenti.unimi.it}

\address{Universit\`a degli Studi di Milano, Dipartimento di Matematica ``F. Enriques'', Via Cesare Saldini 50, 20133 Milano, Italy. \newline \indent Charles University, Faculty of Mathematics and Physics, Department of Mathematical Analysis, Sokolovsk\'{a} 83, 186 75 Praha 8, Czech Republic.}

\email{jacopo.somaglia@unimi.it}
\begin{abstract}
\noindent
We present a characterization of spaces of strictly decreasing functions on trees in terms of bisequentiality. This characterization answers Questions 6.1 and 6.2 of \cite{Ciesla}. Moreover we study the relation between these spaces and the classes of Corson, Eberlein and uniform Eberlein compacta.
\end{abstract}
\maketitle
\section{Introduction}

A sequence of non-empty subsets $\{A_n\}_{n\in \omega}$ of a topological space $X$ converges to a point $x\in X$ if for any neighborhood $U$ of $x$ there exists $n_0\in\omega$ such that $A_n\subset U$ for any $n>n_0$. An ultrafilter $\U\subset \mathcal{P}(X)$ converges to $x\in X$ if every neighborhood $U$ of $x$ is an element of $\U$. A topological space $X$ is \textit{bisequential at a point} $x_0\in X$ if every ultrafilter $\U$ in $X$ convergent to $x_0$ contains a decreasing sequence $\{U_n\}_{n\in\omega}$ converging to $x_0$. $X$ is said \textit{bisequential} if it is bisequential at each $x\in X$. A compact space is said \textit{Eberlein} $\langle$resp. \textit{uniform Eberlein}$\rangle$ if it embeds homeomorphically into some Banach $\langle$resp. Hilbert$\rangle$ space with its weak topology.
Uniform Eberlein compact spaces are closely related to bisequentiality, since every uniform Eberlein compact space with weight less than the first measurable cardinal is bisequential. Since the class of uniform Eberlein compacta is strictly contained within the class of Eberlein compacta (the first example can be found in \cite{BenStar}), one may wonder if the same result still holds by weakening the hypothesis and considering Eberlain compact spaces. Unfortunately this is not the case.
In \cite{Ciesla}, T. Cie\'{s}la introduced a $\sigma$-ideal on $\omega_1\times \omega_1$ in order to prove that the space of strictly decreasing functions on $\omega_1$ in itself, denoted by $\X$ or $\X_{\omega_1}$, is not bisequential (this fact was announced in \cite{Nyikos3} by P. Nyikos).
As a corollary the author obtained an alternative proof of the fact that $\X$ is not uniform Eberlein: this fact was first proved by A. G. Leiderman and G. A. Sokolov in \cite{LeidSoko} by using adequate families of sets. 
Since trees are the natural generalization of ordinal numbers, in the final section of \cite{Ciesla}, the author raises the two following questions:
\begin{itemize}
\item \textit{For which trees T of height less or equal than $\omega_1$ is $\X_T$ bisequential?}
\item \textit{Let T be an Aronszajn tree. Is $\X_T$ bisequential?}
\end{itemize}
Where $\X_T$ denotes the space of strictly decreasing functions on a tree $T$ in itself with totally ordered domain. In this paper we provide answers to both questions as follows.\\
In the remaining part of the introductory section notation and basic notions addressed in this paper are given.
In section 2 we recall the basic notions on trees and we introduce the spaces of strictly decreasing $\langle$strongly strictly decreasing$\rangle$ functions denoted by $\Y_T$ $\langle$respectively, $\X_T\rangle$.
In section 3 we provide a characterization for topological spaces to be bisequential at some point, in terms of the existence of a certain ideal on a local subbase of the point. 
Besides we answer the two aforementioned questions. 
In section 4 we investigate the spaces $\Y_T$ in terms of bisequentiality, proving that the only bisequential $\Y_T$ spaces are the ones with countable $T$.  \\
We denote with $\omega$ the set of natural numbers with the usual order. Given a set $X$ we denote by $|X|$ its cardinality, by $\Fin(X)$ the ideal of all finite subsets, by $\mathcal{P}(X)$ the power set of $X$. Given a subset $A$ of $X$ we denote the characteristic function of $A$ by $1_A$. All the topological spaces are assumed to be Hausdorff and completely regular. We refer to \cite{Jech} for basic definitions and results in set theory that are used here without a specific reference.\\
We conclude the introductory section recalling some basic notions about the classes of Corson, Eberlein, uniform Eberlein compacta and adequate families of subsets. We refer to \cite{Negre} for an old but still good survey in this area.\\
A compact space is said \textit{Corson} if it homeomorphically embeds into a $\Sigma$-product
\begin{equation*}
\Sigma(\Gamma)=\{x\in \R^{\Gamma}:|\{\gamma\in \Gamma:x(\gamma)\neq 0\}|\leq \omega_0\}.
\end{equation*}
endowed with the subspace topology inherited by $\R^{\Gamma}$, with the product topology, for some set $\Gamma$.\\
The notion of adequate families of subsets was introduced by Talagrand in \cite{Tal}. 
\begin{defn}
Let $X$ be a non-empty set, a family $\mathcal{A}$ of subsets of $X$ is called \textit{adequate} if:
\begin{enumerate}[$(i)$]
\item If $A\in \mathcal{A}$ and $B\subset A$, then $B\in \mathcal{A}$.
\item If $B\subset X$ is such that all finite subsets of $B$ belongs to $\mathcal{A}$, then $B\in \mathcal{A}$.
\end{enumerate}
\end{defn}
Every adequate family of subsets of $X$ can be viewed as a closed subset of $\{0,1\}^X$. Therefore given an adequate family $\mathcal{A}$, we call the space
\begin{equation*}
K_{\mathcal{A}}=\{1_A\in \{0,1\}^{X}: A\in\mathcal{A}\}
\end{equation*}
an \textit{adequate compact}. It follows immediately that $K_{\mathcal{A}}$ is a Corson compact space whenever $\mathcal{A}$ consists of countable sets. A characterization of adequate compact spaces in terms of Eberlein and uniform Eberlein classes is due to Leiderman and Sokolov \cite{LeidSoko}.

\begin{thm}[\cite{LeidSoko}]\label{FamAdeguate}
Let $X$ be a set and $\mathcal{A}$ be an adequate family of subsets of $X$. Then the following assertions hold:
\begin{enumerate}
\item $K_{\mathcal{A}}$ is an Eberlein compact space if and only if there is a partition $X=\bigcup_{i\in\omega}X_i$ such that $|\supp(x)\cap X_i|<\aleph_0$ for each $x\in K_{\mathcal{A}}$ and $i\in\omega_0$.
\item $K_{\mathcal{A}}$ is an uniform Eberlein compact space if and only if there is a partition $X=\bigcup_{i\in\omega}X_i$ and an integer-valued function $N(i)$ such that $|\supp(x)\cap X_i|<N(i)$ for each $x\in K_{\mathcal{A}}$ and $i\in\omega_0$.
\end{enumerate}
\end{thm}

\section{Spaces of decreasing functions on trees}\label{decfunctrees}

In this section we deal with set-theoretical trees, for this reason we fix the basic definitions and notations useful for the second part.\\ 
A \textit{tree} is a partially ordered set $(T,<)$ such that the set of predecessors $\{s\in T:s<t\}$ of any $t\in T$ is well-ordered by $<$. For any element $t\in T$, $\htte(t,T)$ denotes the order type of $\{s\in T:s<t\}$ and $\suc t=\{s\in T:s\geq t\}$ denotes the set of the successors of $t$. For any two elements $s,t\in T$, we say that $s,t$ are \textit{comparable} if $s<t \lor s>t \lor s=t$ holds. We say that $s,t\in T$ are \textit{incomparable}, and it is denoted by $s\perp t$, if they are not comparable. Given two subsets $A,B\subset T$ we say that $A,B$ are incomparable, denoted by $A\perp B$, if $s$ and $t$ are incomparable for any choice of $s\in A$ and $t\in B$. \\
For any ordinal $\alpha$, the set $\Lev_{\alpha}(T)=\{t\in T:\htte(t,T)=\alpha\}$ is called the $\alpha$th\textit{ level} of $T$. The height of $T$ is denoted by $\htte(T)$, and it is the least $\alpha$ such that $\Lev_{\alpha}(T)=\emptyset$. We denote by $T_{\beta}=\bigcup_{\gamma < \beta}\Lev_{\gamma}(T)$ and by $T_{>\beta}=\bigcup_{\gamma > \beta}\Lev_{\gamma}(T)$ for any ordinal $\beta$.\\
We say that a subset $C\subset T$ is a \textit{chain} if $C$ is totally ordered. A maximal chain is said \textit{branch}. A subset $A\subset T$ is said \textit{antichain} if $x\perp y$ for each $x,y\in A$. We recall that a tree $T$ of height $\omega_1$ is an \textit{Aronszajn tree} if all its levels and all its branches are at most countable,  furtermore if $T$ has no uncountable antichains $T$ is said \textit{Suslin tree}. An Aronszajn tree $T$ is said \textit{special} if it is the union of countable its antichains. It is a well-known fact that an uncountable subset of a Suslin tree is a Suslin tree with the inherited order. Moreover, we recall that if $T$ is a Sulslin tree and $T_1=\{t\in T: |\suc t|\leq\aleph_0\}$, we have that $T\setminus T_1$ is uncountable. In fact, the set of minimal elements of $T_1$ is an antichain of $T$, therefore $T_1$ can be written as a countable union of countable set.\\
Let $T$ be a tree, as in \cite{Ciesla} we consider the subspace $\X_T$ of the Cantor cube $\{0,1\}^{T\times T}$ defined as follows:
\begin{equation*}
\X_T=\{1_F\in \{0,1\}^{T\times T}: F\mbox{ is the graph of a strongly strictly decreasing function}\}.
\end{equation*}
Where by \textit{strongly strictly decreasing function} we mean a partial function $f:T\to T$ with totally ordered domain. We observe that $\X_T$ is an adequate compact and moreover, since the support of each element of $\X_T$ is finite, we have that the spaces $\X_T$ are Eberlein compacta, for any choice of $T$. We are going to provide a sufficient condition for $\X_T$ to be uniform Eberlein.
\begin{prop}
Let $T$ be a tree. Suppose that $T=\bigcup_{i\in\omega}A_i$, where $A_i$ are antichains. Then $\X_T$ is a uniform Eberlein space.
\end{prop}
\begin{proof}
We are are going to use the Theorem \ref{FamAdeguate}. Let $T$ be a tree as in the hypothesis, we may suppose that $T=\bigcup_{i\in\omega} A_i$, where $A_i$ are disjoint antichains. For each $i\in \omega$, we define 
\begin{equation*}
X_i=\bigcup_{i\in\omega} A_i\times T.
\end{equation*}
Clearly the family $\{X_i\}_{i\in \omega}$ is a partition of $T\times T$ and for each $x\in \X_T$ we have $|\supp(x)\cap X_i|\leq1$. Therefore $\X_T$ is a uniform Eberlein compact space.
\end{proof}
In particular for $T$ special Aronszajn tree, $\X_T$ is uniform Eberlein. Recalling \cite[Theorem 5.2]{Ciesla} we have an easy way to show that if $T$ is a special Aronszajn tree, then $\X_T$ is bisequential. Unfortunately this argument seems hard to generalize to non-special Aronszajn trees. 
Thus in order to study the bisequentiality of this class of trees and to answer Question 6.2 of \cite{Ciesla}, in the next section we will work directly on bisequentiality, providing a complete characterization of those trees which have the property.\\
In a similar fashion we define another class of spaces of decreasing function on trees. Let $T$ be a tree and $f:T\to T$ be a partial function, we say that $f$ is \textit{strictly decreasing} if for every two comparable elements $s<t$ of the domain of $f$, we have $f(t)<f(s)$. We consider the subspace $\Y_T$ of the Cantor cube $\{0,1\}^{T\times T}$ defined as follows:
\begin{equation*}
\Y_T=\{1_F\in\{0,1\}^{T\times T}: F \mbox{ is a graph of a strictly decreasing function}\}.
\end{equation*}
Notice that, for each tree $T$, we have $\X_T\subset \Y_T$ and moreover $\X_T=\Y_T$ if and only if $T$ is a linear order. Also in this case it is easy to show that, for any tree $T$, the space $\Y_T$ is an adequate compact space. Moreover we observe that, given a tree $T$, the space $\Y_T$ is a Corson compact space if and only if every antichain of $T$ is countable.

\begin{prop}\label{PropEberYT}
Let $T$ be a tree, then $\Y_T$ is an Eberlein compact space if and only if $T=\bigcup_{i\in\omega} T_i$ such that $T_i$ has only finite antichains for each $i\in\omega$.
\end{prop}

\begin{proof}
Assume that $\Y_T$ is an Eberlein compact space. Then there exists a countable partition $\{A_i\}_{i\in\omega}$ of $T\times T$ as in Theorem \ref{FamAdeguate}. Let us define, for each $i\in\omega$, $T_i=\pi_1 [A_i]$, where $\pi_1$ is the canonical projection onto the first coordinate. Then it is clear that each $T_i$ has only finite antichains. Indeed suppose by contradiction that there exists an infinite antichain $C\subset T_i$, hence it is possible to define a strictly decreasing function  $f:C\to T$ by $f(x)=y$ with $(x,y)\in A_i$. Denoting by $F$ the graph of $f$, we get $|\supp(1_F)\cap A_i|\geq\aleph_0$, which is a contradiction.\\
Conversely, assume that $T=\bigcup_{i\in\omega} T_i$ and each $T_i$ has finite antichains only. Without loss of generality, we may assume that the family $\{T_i\}_{i\in \omega}$ is pairwise disjoint. For each $i\in\omega$ consider $A_i=T_i\times T$. The family $\{A_i\}_{i\in\omega}$ is a partition of $T\times T$, which satisfies the requirement of Theorem \ref{FamAdeguate}. In fact, let $f:A\subset T_i\to T$ be a strictly decreasing function, we observe that $A$ is a subset with finite chains and antichains. Therefore $A$, and in particular the graph of $f$, is finite.
\end{proof}

Notice that the previous characterization can be rewritten as: $\Y_T$ is an Eberlein compact space if and only if $T=N\cup\bigcup_{i\in \alpha}C_i$, where $\alpha\leq \omega$, $C_i$ are uncountable chains and $N$ is a countable subset of $T$. In fact, suppose that $\Y_T$ is an Eberlein compact space. We may suppose that $T$ has uncountable size, otherwise the assertion would be clear. By Proposition \ref{PropEberYT}, we have $T=\bigcup_{i\in\omega}T_i$, where $T_i$ has only finite antichains. It is a well-known fact that an uncountable tree with all finite levels, has an uncountable branch. Therefore, suppose that some uncountable $T_i$ contains infinitely many uncountable branches, then, since $T_i$ has uncountable height, we would have an infinite level in $T_i$. That is a contradiction. It follows that $T_i$ has finitely many uncountable branches, which gives the assertion. The other implication easily follows by Proposition \ref{PropEberYT}.\\
Let $T$ be a Suslin tree, from the previous observation it follows that the space $\Y_T$ is not Eberlein, Moreover since each antichain of $T$ is at most countable, we have that $\Y_T$ is a Corson compact space. We refer to \cite{Negre} for other examples of Corson compact spaces that are not Eberlein.\\
In section \ref{SectionYT} we will provide a full characterization of uniform Eberlein $\Y_T$ spaces.

\section{Bisequentiality on $\X_T$ spaces}

The aim of this section is to characterize the spaces $\X_T$ in terms of bisequentiality. Let us start by stating two necessary conditions for $\X_T$ spaces to be bisequential.\\   
Since bisequentiality is preserved passing to subspaces (see \cite[Proposition 3.D.3]{Michael}), it follows that by Cie\'{s}la's result \cite[Theorem 1.2]{Ciesla}, $\X_T$ cannot contain $\X_{\omega_1}$ as a subspace, therefore $T$ must have no uncountable branches.
Moreover, we notice that in \cite[Example 10.15]{Michael} it was proved that the one-point compactification of a discrete set $\Gamma$ is bisequential if and only if the cardinality of $\Gamma$ is less than the first measurable cardinal. We recall that a cardinal $\kappa$ is said \textit{measurable} if there exists a $\kappa$-complete nonprincipal ultrafilter on $\kappa$.  It is easy to show that if $S$ is an antichain of $T$, then $\X_T$ contains a copy of the one-point compactification of a discrete set of size $|S|$. Therefore, if the cardinality of $S$ is greater or equal to the first measurable cardinal, the space $\X_T$ is not bisequential.\\ 
The idea behind the proof of \cite[Theorem 1.2]{Ciesla} is to show that the existence of a particular $\sigma$-ideal implies the non-bisequentiality of the space $\X_{\omega_1}$ at the empty  function. Since not being bisequential at some point implies not being bisequential at all, the result follows.\\
In the first part of this section we generalize this argument. First we show that by weakening the requirement of being a $\sigma$-ideal, one can prove that the bisequentiality of a topological space at a point is actually equivalent to the existence of an ideal (with certain properties) on a local subbase of the point. 
Secondly, we show that every space $\X_T$ is bisequential if and only if it is bisequential at the empty function.
In the final part of the section, combining this two results we show that the two necessary conditions stated above are also sufficient for the bisequentiality, that is:

\begin{thm}\label{ThmCarattXT}
Let $T$ be a tree. $\X_T$ is a bisequential space if and only if $T$ satisfies the following conditions:
\begin{itemize}
\item $T$ has size less than the first measurable cardinal.
\item $T$ has no uncountable branches.
\end{itemize}
\end{thm}

Notice that since every measurable cardinal is regular, 
we have that each tree $T$ of height less than or equal to $\omega_1$ has cardinality strictly less than the first measurable cardinal if and only if each level of $T$ has cardinality strictly less than the first measurable cardinal.\\
We recall that a family $\F$ of sets has the \textit{finite intersection property} if every finite $\mathcal{G}=\{X_1,...,X_n\}\subset \F$ has nonempty intersection $X_1\cap\dots \cap X_n\neq\emptyset$.

\begin{thm}\label{ThmCarattBiseq}
Let $X$ be a topological space. Let $\{V_{\alpha}\}_{\alpha\in A}$ be a local subbase for a point $x\in X$. Then $X$ is not bisequential at $x$ if and only if there exists an ideal $\I$ on $A$ satisfying the following conditions:
\begin{enumerate}[$(i)$]
\item $\{\alpha\}\in \I$ for every $\alpha\in A$;
\item the family $\mathcal{F_{\I}}=\{\bigcap_{\alpha\in S}V_{\alpha}: S\in \I\}\cup \{X\setminus \bigcap_{\alpha\in S}V_{\alpha} : S\in \mathcal{P}(A)\setminus \I\} $ has the finite intersection property;
\item if $\{S_i\}_{i\in\omega}\subset\I$, then $\bigcup_{i\in\omega}S_i\neq A$.
\end{enumerate}
\end{thm}

\begin{proof}
$(\Rightarrow)$ Suppose that $X$ is not bisequential at $x$, therefore there exists an ultrafilter $\U$ converging to $x$ such that every countable decreasing family $\{U_i\}_{i\in\omega}\subset \U$ does not converge to $x$.\\
Let us define a family of subsets of $A$ as $\I=\{S\subset A: \bigcap_{\alpha\in S}V_{\alpha}\in \U\}$. Since $\U$ is an ultrafilter on $X$, it is easy to prove that $\I$ is an ideal on $A$. First, since $\U$ is non-principal, then $A\notin \I$. Let $S_1,S_2\in \I$, then 
\begin{equation*}
\bigcap_{\alpha\in S_1\cup S_2}V_{\alpha}=(\bigcap_{\alpha\in S_1}V_{\alpha})\cap (\bigcap_{\beta\in S_2}V_{\beta})\in \U.
\end{equation*}
Hence $S_1\cup S_2\in \I$. Further, if $S_1\subset S_2$ and $S_2\in\I$, we have $\bigcap_{\alpha\in S_2}V_{\alpha}\subset \bigcap_{\alpha\in S_1}V_{\alpha}$, hence $S_1\in \I$. It remains to prove that $\I$ satisfies $(i)-(iii)$:
\begin{enumerate}[$(i)$]
\item since $\U$ converges to $x$, we have $V_{\alpha}\in \U$ for every $\alpha\in A$;
\item $\mathcal{F}_{\I}$ has the finite intersection property, since it is extended by $\U$;
\item suppose by contradiction that there exists a countable family $\{S_i\}_{i\in\omega}\subset \I$, such that $\bigcup_{i\in\omega}S_i=A$. Let $U_i=\bigcap_{j\leq i}\bigcap_{\alpha\in S_{j}}V_{\alpha}$, since $S_1\cup\dots \cup S_i\in \I$ we have $U_i\in \U$ for every $i\in \omega$. The family $\{U_i\}_{i\in\omega}\subset \U$ is a countable decreasing family that converges to $x$. That is clearly a contradiction.
\end{enumerate}
$(\Leftarrow)$ Since $\mathcal{F}_{\I}$ has the finite intersection property it can be extended to an ultrafilter $\U$ on $X$. Since $\{\alpha\}\in\I$ for every $\alpha\in A$, the ultrafilter $\U$ converges to $x$.\\
Suppose by contradiction that there exists a countable decreasing family $\{U_i\}_{i\in\omega}\subset \U$ convergent to $x$. Let us define
\begin{equation*}
A_i=\{\alpha\in A: U_i\subset V_{\alpha}\},
\end{equation*}
for every $i\in\omega$. It follows that $U_i\subset \bigcap_{\alpha\in A_{i}}V_{\alpha}\in\U$. Therefore, since $\U$ is an ultrafilter, we have $A_i\in \I$, otherwise we would have both $\bigcap_{\alpha\in A_i}V_{\alpha}\in \U$ and $(X\setminus\bigcap_{\alpha\in A_i}V_{\alpha})\in\U$. Since $\{U_{i}\}_{i\in\omega}$ converges to $x$, for every $\alpha\in A$ there exists $i\in\omega$ such that $U_i\subset V_{\alpha}$. Hence $\bigcup_{i\in\omega}A_i=A$, which contradicts the third property of $\I$.
\end{proof}

 We will use $O$ to denote the graph of the empty function. Given $1_F\in \X_T$, we set 
\begin{equation*}
\begin{split}
W_F&=\{1_G\in \X_{T}: F\subset G\},\\
W^F&=\{1_G\in \X_T: G\cap F=\emptyset,\,  1_{F\cup G}\in \X_T \}.
\end{split}
\end{equation*}
We observe that
\begin{equation*}
W_F=\bigcap_{(s,t)\in F}\{1_G\in\X_T: 1_G(s,t)=1\},
\end{equation*}
since $F$ is a finite subset of $T\times T$, we have that $W_F$ is a clopen subset of $\X_T$.

\begin{prop}\label{LemEmptyFunction}
$\X_T$ is not a bisequential space if and only if $\X_T$ is not bisequential at $1_{O}$.
\end{prop}

We need first a couple of lemmata. For sake of completeness we provide also the proofs of the following well-known results.

\begin{lem}\label{UltSub}
Let $X$ be a set and $\U$ be an ultrafilter on it. Suppose that $Y\in \U$, let
\begin{equation*}
\U_{Y}=\{U\cap Y:U\in \mathcal{U}\},
\end{equation*}
then $\U_Y$ is an ultrafilter on $Y$ contained in $\U$.
\end{lem}

\begin{proof}
Since $Y\in \U$, we have $U\cap Y\in \U$ for every $U\in \U$. Hence $\U_Y\subset \U$. We observe that $\emptyset \notin \U_{Y}$. We want to show that for every $A\subset Y$, we have either $A\in \mathcal{U}_{Y}$ or $Y\setminus A\in \U_{Y}$. We consider $A\cup (X\setminus Y)$, since $\U$ is an ultrafilter on $X$ we have either $A\cup (X\setminus Y)\in \U$ or $Y\setminus A \in \U$. In the first case we have $\U_Y\ni Y\cap (A\cup (X\setminus Y))=A$, in the second case $Y\setminus A\in \U_Y $. This concludes the proof.
\end{proof}

It is well-known that subspaces of a bisequential space are bisequential \cite[Proposition 3.D.3]{Michael}. In fact can be said something stronger.

\begin{lem}\label{PropSubspace}
Let $X$ be a topological space and $Y$ be a subspace of $X$, and let $y\in Y$.
\begin{enumerate}[$(a)$]
\item If $Y$ is not bisequential at $y$, then $X$ is not bisequential at $y$.
\item If $Y$ is an open subspace of $X$, then $Y$ is bisequential at $y$ if and only if $X$ is bisequential at $y$.
\end{enumerate}
\end{lem}

\begin{proof}
$(a)$ Let $\U_{Y}$ be an ultrafilter on $Y$ converging to $y\in Y$, such that every countable decreasing subfamily $\{U_{i}\}_{i\in \omega}\subset \U_{Y}$ does not converge to $y$.\\
Since $Y\subset X$ and $\U_{Y}$ has the finite intersection property, $\U_{Y}$ can be extended to an ultrafilter $\tilde{\U}_{Y}$ on $X$. Let $V_{y}$ be a neighborhood of $y$, then $V_{y}\cap Y\in \U_{Y}$, hence $V_{y}\in \tilde{\U}_{Y}$. Therefore $\tilde{\U}_{Y}$ converges to $y$ in $X$.\\
Suppose by contradiction that there exists a countable decreasing family $\{U_i\}_{i\in\omega}\subset \tilde{\U}_{Y}$ convergent to $y$ in $X$. Since for every $i\in\omega$ we have $Y\cap U_i\in \U_{Y}$ (otherwise we would have $Y\setminus U_{i}\in \U_Y \subset \tilde{\U}_{Y}$ and $U_i \in \tilde{\U}_{Y}$, but $(Y\setminus U_{i})\cap U_i=\emptyset$) and for every neighborhood $V_y$ of $y$ we have $Y\cap U_i\subset Y\cap V_y$ eventually, we obtain a contradiction.\\
$(b)$ It is enough to show that if $X$ is not bisequential at $y$, then $Y$ is not bisequential at $y$. Thus let us assume that there exists an ultrafilter $\U$ on $X$ convergent to $y$ such that every countable decreasing chain $\{U_i\}_{i\in\omega}\subset \U$ is not convergent to $y$. Since $Y$ is an open subset of $X$ and $y\in Y$, we have $Y\in\U$. We apply Lemma $\ref{UltSub}$ getting an ultrafilter $\U_Y$ on $Y$, defined by $\U_Y=\{U\cap Y:U\in \mathcal{U}\}\subset \U$. Since every open set in $Y$ is open in $X$, the ultrafilter $\U_Y$ converges to $y$. Moreover, since every countable chain in $\U_Y$ belongs to $\U$, it follows that $Y$ is not bisequential at $y$.
\end{proof}

\begin{lem}\label{PropHomeo}
Let $X$, $Y$ be topological spaces, and suppose that $\varphi:X\to Y$ is a homeomorphism between them. Then $X$ is bisequential at $x$ if and only if $Y$ is bisequential at $\varphi(x)$.
\end{lem}

\begin{proof}
Let $\U$ be an ultrafilter on $X$ convergent to $x$, such that every countable decreasing subfamily $\{U_{i}\}_{i\in \omega}\subset \U$ does not converge to $x$. Define $\U_Y=\varphi(\U)=\{\varphi(U): U\in \U\}$. This is an ultrafilter on $Y$, convergent to $\varphi(x)$ by definition of homeomorphism. If $\{U_i\}_{i\in\omega}$ is a decreasing chain in $\U_Y$, then $\{\varphi^{-1}(U_i)\}_{i\in\omega}$ is a decreasing chain in $\U$ and by assumption there exists $V\subset X$ open neighborhood of $x$ such that $\varphi^{-1}(U_i)\nsubseteq V$ for every $i$, or equivalently, $U_i\nsubseteq \varphi(V)$ for every $i$. But $\varphi(V)$ is open in $Y$. This proves that $Y$ is not bisequential at $\varphi(x)$.\\
The converse follows combining the previous argument with the fact that $\varphi^{-1}:Y\to X$ is a homeomorphism too.
\end{proof}

\begin{proof}[Proof of Proposition \ref{LemEmptyFunction}]
The "only if part" is trivial. Thus let us assume that $\X_T$ is not bisequential, in particular $\X_T$ is not bisequential at some point $1_F\in \X_T$. 

Consider the following mapping
\begin{equation*}
\begin{split}
\varphi:& W_F\to W^F\\
&1_G\mapsto 1_{G\setminus F}
\end{split}
\end{equation*}
\textbf{Claim:} $\varphi$ is a homeomorphism.\\
Hence, since the space $W_F$ is an open subset of $\X_T$ and $1_F\in W_F$, applying Lemma \ref{PropSubspace} we get $W_F$ is not bisequential at $1_F$, then by Lemma $\ref{PropHomeo}$ $W^F$ not bisequential at $1_O=\varphi(1_F)$. Therefore, by Lemma $\ref{PropSubspace}$, we get $\X_T$ not bisequential at $1_O$, that is our conclusion.\\
It remains to prove the claim. The map $\varphi$ is clearly well-defined and continuous. Moreover $\varphi$ is injective and surjective, and since $W_F$ is a closed subspace of $\X_T$ and $\X_T$ is a compact space, $W_F$ is compact and $\varphi$ is a closed map. This concludes the proof.
\end{proof}

We are now ready to prove the main result of this section.

\begin{proof}[Proof of Theorem \ref{ThmCarattXT}]
We showed at the very beginning of the section that the two stated conditions are necessary, hence the "only if" part follows.
For this reason we assume that $T$ is a tree of cardinality less than the first measurable cardinal and each of its branches has countable cardinality.\\
In order to prove that the space $\X_T$ is bisequential we are going to combine Theorem \ref{ThmCarattBiseq} with Lemma \ref{LemEmptyFunction}. Let $\{V_{(s,t)}\}_{(s,t)\in T\times T}$ be the local subbase of $1_O$ defined as follows
\begin{equation*}
V_{(s,t)}=\{1_F\in\X_T: (s,t)\notin F\},
\end{equation*}
for any choice of $(s,t)\in T\times T$. Let $\I$ be an ideal on $T\times T$ that satisfies conditions $(i)$ and $(ii)$ of Theorem \ref{ThmCarattBiseq}. We are going to prove that $\I$ does not satisfy the condition $(iii)$.\\
At first we notice that for any $A,B\subset T$ incomparable, we have
\begin{equation}\tag{$*$}\label{IncomparableDom}
(\X_T\setminus \bigcap_{(r_1,s_1)\in A\times T} V_{(r_1,s_1)})\cap (\X_T\setminus \bigcap_{(r_2,s_2)\in B\times T}V_{(r_2,s_2)}) =\emptyset,
\end{equation}
since by definition of $\X_T$ every function has totally ordered domain. Since $\F_{\I}$ has the finite intersection property, we have that at least one between $A\times T$ and $B\times T$ belongs to $\I$. The same argument can be used to prove that
\begin{equation}\tag{$**$}\label{IncomparableCodom}
(\X_T\setminus \bigcap_{(r_1,s_1)\in T\times A} V_{(r_1,s_1)})\cap (\X_T\setminus \bigcap_{(r_2,s_2)\in T\times B}V_{(r_2,s_2)}) =\emptyset.
\end{equation}
since every decreasing function with totally ordered domain has also totally ordered codomain.
Consider for every $t\in T$ and $A\subset T$ the sets
\begin{equation*}
\begin{split}
S_t&=(\uparrow t)\times T,\\
S_A&=\bigcup_{t\in A} S_t.
\end{split}
\end{equation*}
Suppose that for some $t\in T$ we have $S_{t}\notin \I$, then if $A\subset T$ incomparable with $t$, we have $S_A\in \I$ by what stated before. Furthermore, if $S_t\in \I$ and $t\leq s$, we have $S_s\subset S_t$, hence $S_s\in \I$. Therefore the subset
\begin{equation*}
Z_1=\{t\in T:S_t\notin \I\}
\end{equation*}
is totally ordered and if $t\in Z_1$, then $s\in Z_1$ for every $s\leq t$. Therefore $Z_1$ is countable. Let $\beta<\omega_1$ be the least countable ordinal such that $Z_1\subset T_{\beta}$. Let $A\subset \Lev_{\alpha}(T)$, for some $\alpha<\omega_1$, we set $c(A)=\Lev_{\alpha}(T)\setminus A$. From what we have observed in the first part of the proof, we obtain $S_{c(t)}\in \I$, for every $t\in Z_1$. In particular we have $(T_{\beta}\setminus Z_1)\times T=\bigcup_{t\in Z_1}S_{c(t)}$. 
Now we are going to cover the set $(T\setminus T_{\beta})\times T$ with a countable subfamily of $\I$. In order to do that we define a family $\I_{\beta}$ on $\Lev_{\beta}(T)$ in the following way: $A\in \I_{\beta}$ if and only if $S_{A}\in \I$. If $\Lev_{\beta}(T)\in \I_{\beta}$, then $S_{\Lev_{\beta}(T)}=T\setminus T_{\beta}$, hence we have covered $(T\setminus Z_1)\times T$ with countable elements of $\I$. 
Thus let us assume that $\Lev_{\beta}(T)\notin \I_{\beta}$, hence the family $\I_{\beta}$ is an ideal, and it is maximal by (\ref{IncomparableDom}), since $A$ and $c(A)$ are incomparable, so at least one among $S_{A}$ and $S_{c(A)}$ belongs to $\I$. 
We observe that, since $Z_1\cap\Lev_{\beta}(T)=\emptyset$, we have $\Fin(\Lev_{\beta}(T))\subset \I_{\beta}$, therefore $\I_{\beta}$ is nonprincipal. Hence since the cardinality of $\Lev_{\beta}(T)$ is less than the first measurable cardinal there exists a countable family $\{A_i\}_{i\in\omega}\subset \I_{\beta}$, such that $\Lev_{\beta}(T)=\bigcup_{i\in\omega}A_i$. Thus we obtain $(T\setminus T_{\beta})\times T= \bigcup_{i\in\omega}S_{A_i}$. Thus we obtained a countable subfamily of $\I$ covering $(T\setminus Z_1)\times T$.\\
We repeat the argument with
\begin{equation*}
\begin{split}
S^t&=T\times(\uparrow t),\\
S^A&=\bigcup_{t\in A} S^t.
\end{split}
\end{equation*}
Hence we obtain a countable subfamily of $\I$ covering $T\times (T\setminus Z_2)$, where $Z_2=\{t\in T:S^t\notin \I\}$. Hence since $Z_1\times Z_2$ is countable and $\Fin(T\times T)\subset \I$, we have that $\I$ does not satisfy condition $(iii)$. Therefore $\X_T$ is a bisequential space.
\end{proof}

As an immediate corollary of the previous theorem we obtain the following result, which answers the Question 6.2 of \cite{Ciesla}.

\begin{cor}
Let $T$ be an Aronszajn tree. Then $\X_T$ is a bisequential space.
\end{cor}

\section{Bisequentiality on $\Y_T$ spaces}\label{SectionYT}

In Section \ref{decfunctrees} we have introduced, for a tree $T$, the class of $\Y_T$ spaces (we recall that $\Y_T$ is the space of strictly decreasing function on $T$). The core of this section is to characterize $\Y_T$ spaces in terms of bisequentiality. It turns out that the only trees for which $\Y_T$ is bisequential are the countable ones.\\ 
We recall that for every tree $T$ we have $\X_T \subset \Y_T$. Hence, since the bisequentiality is preserved by taking subspaces, by Theorem \ref{ThmCarattXT} we can restrict our attention on trees with no uncountable branches and with cardinality strictly less than the first measurable cardinal. Let us consider a tree $T$ with an uncountable antichain $C$. We may suppose that $|C|=\omega_1$, hence we can enumerate it as $C=\{t_{\alpha}\}_{\alpha<\omega_1}$. We observe that for every $\beta\leq \omega_1$, the partial function $f_{\beta}:T\to T$ defined on $\{t_{\alpha}\}_{\alpha<\beta}$ as $f_{\beta}(t_{\alpha})=t_{\alpha}$, for $\alpha<\beta$, is strictly decreasing. Hence, denoting by $F_{\beta}$ the graph of the partial function $f_{\beta}$, we get that the transfinite sequence $\{1_{F_{\beta}}\}_{\beta\leq\omega_1}$ is a homeomorphic copy of $[0,\omega_1]$ into $\Y_T$. Whence, since $[0,\omega_1]$ is not a bisequential space we get that the space $\Y_T$ is not bisequential as well.\\
Let $T$ be a tree of countable size. Hence, since $T$ has at most countable antichains, we get that $\Y_T$ is a Corson compact space. Moreover, since the set of all finitely supported elements of $\Y_T$ is countable and dense in $\Y_T$, we get that $\Y_T$ is a separable Corson compact space, hence it is a metrizable space, in particular it is bisequential.\\
Let us finally consider the case when $T$ is a Suslin tree. Let us start by proving an useful lemma. We limit ourself to Suslin trees, since this is the only case left to be studied. 

\begin{lem}\label{LemmaNonconfrontabili}
Let $T$ be a Suslin tree and $P_1,...,P_n$ be a finite family of uncountable subsets of $T$. Then there exists $(t_1,...,t_n)\in P_1\times...\times P_n$ such that $t_i\perp t_j$ for any $i\neq j$.
\end{lem}

\begin{proof}
We may suppose without loss of generality that every uncountable subset $S\subset T$ considered in the present proof satisfies the following property: for every $x\in S$, we have $|\suc x\cap S|>\aleph_0$.\\
We will show by induction, in fact, that for every finite family $P_1,...,P_n$ of uncountable subsets of $T$ there exists a finite family $P'_1,...,P'_n$ of disjoint uncountable subsets of $T$ such that $P'_i\subset P_i$ for every $i=1,...,n$ and $P_{i}^{'}\perp P_{j}^{'}$ whenever $i\neq j$.\\
In order to do that, let us consider $n=2$: let $x,y\in P_1$ such that $x\perp y$, since $P_1$ is a Suslin tree such elements exist, and set $\beta=\min(\htte(x, T),\htte(y, T))$. Hence there exists $z\in \Lev_\beta(T)$ such that $P_2\cap \suc z$ is an uncountable subset of $P_2$. Then, we obtain $x\perp z$ or $y\perp z$. Suppose without loss of generality $x\perp z$, then $P'_1=P_1\cap \suc x$ and $P'_2=P_2\cap \suc z$ are as required.\\
Finally, let $n\geq 2$ and assume that for any family $P_1,...,P_n$ the thesis holds. Let $P_1,...,P_{n+1}$ be any family of size $n+1$, we get the family $P'_1,...,P'_{n}$ by applying the induction hypothesis to $P_1,...,P_{n}$. Repeating the argument with $P'_2,...,P'_n, P_{n+1}$ one gets $P''_2,...,P''_n, P'_{n+1}$ and then again with $P'_1,P'_{n+1}$ to get $P''_1,P''_{n+1}$. Therefore
$P''_1,...,P''_{n+1}$ is the desired family.
\end{proof}

Let $T$ be a tree and $A\subset T\times T$. Let $s,t\in T$, we set the sections of $A$ as follows: 

\begin{equation*}
\begin{split}
&A_t=\{s\in T:(t,s)\in A\};\\
&A^s=\{t\in T:(t,s)\in A\}.
\end{split}
\end{equation*}

Let us recall that by $O$ we denote the graph of the empty function.

\begin{prop}\label{biseqSuslin}
Let $T$ be a Suslin tree. Then $\Y_T$ is not a bisequential space.
\end{prop}

\begin{proof}
Let $(t,s)\in T\times T$ and 
\begin{equation*}
V_{(t,s)}=\{1_{F}\in \Y_T:(t,s)\notin F\}.
\end{equation*}
Since the family $\{V_{(t,s)}\}_{(t,s)\in T\times T}$ is a local subbase of $1_O$, we are going to use the Theorem \ref{ThmCarattBiseq}. In order to do that, borrowing the idea from \cite{Ciesla}, let us define a suitable ideal on $T\times T$:
\begin{equation*}
\I= \{A\subset T\times T: \mbox{ for all but countably many }t, |A_t|\leq \aleph_0 \mbox{ and } |A^t|\leq \aleph_0\}.
\end{equation*}
The ideal $\I$ clearly satisfies the conditions $(i)$ and $(iii)$ of Theorem \ref{ThmCarattBiseq}. It remains to prove that the family $\F_{\I}$, defined as in Theorem \ref{ThmCarattBiseq}, has the finite intersection property.\\
In order to do this let $A\in \I$ and $B_1,...,B_n$ subsets of $T\times T$ which do not belong to $\I$. We observe that for each $1\leq j \leq n$ at least one of the following assertions is satisfied:
\begin{itemize}
\item $|(B_j)_t|=\omega_1$ for uncoutably many $t\in T$, in this case let us define $U_j$ as the subset of such elements.
\item $|(B_j)^s|=\omega_1$ for uncoutably many $s\in T$, in this case let us define $W_j$ as the subset of such elements.
\end{itemize}
We may assume without loss of generality that there exists $1\leq m\leq n+1$ such that for $0<j<m$ we have $|(B_j)_t|=\omega_1$ for uncountably many $t\in T$ and for $m\leq j< n+1$ we have $|(B_j)^s|=\omega_1$ for uncountably many $s\in T$.\\
Let us define 
\begin{equation*}
\alpha=\sup\{\htte(t,T):t\in A, |A_t|=\omega_1\},\:\:\:\:  \beta=\sup\{\htte(s,T):s\in A, |A^s|=\omega_1\}.
\end{equation*}
We observe that, since $A\in \I$, both $\alpha$ and $\beta$ are countable. \\
Since every $W_j$ is uncountable, there is $s_j\in W_j\cap T_{>\beta}$, for each $m\leq j\leq n$.
Let $\tilde{\alpha}=\max\{\sup\{\htte(t,T): t\in A^{s_j},m\leq j\leq n\},\alpha\}$, since $\htte(s_j,T)>\beta$ and $m\leq j\leq n$, we have that $\tilde{\alpha}<\omega_1$. Define $P_j=(B_j)^{s_j}\cap T_{>\tilde{\alpha}}$ for all $m\leq j\leq n$, they are uncountable since $s_j\in W_j$.\\ 
For every $0< i< m$ consider $P_i=U_i\cap T_{>\alpha}$, they are uncountable since every $U_i$ is uncountable.
Hence by Lemma \ref{LemmaNonconfrontabili}, there are $t_k\in P_k$ for all $0<k<n+1$ such that $t_k\perp t_{k'}$ for $k\neq k'$. Notice for $m\leq j\leq n$ we have $(t_j,s_j)\in B_j\setminus A$.\\
Finally, let $\tilde{\beta}=\max\{\sup\{\htte(s,T): s\in A^{t_i},0< i< m\},\beta\}$, since $\htte(t_i,T)>\alpha$ and $0<i< m$, we have that $\tilde{\beta}<\omega_1$. Therefore there are $s_i \in (B_i)_{t_i}\cap T_{>\tilde{\beta}}$, for any $0<i<m$. Notice for $0<i< m$ we have $(t_i,s_i)\in B_i\setminus A$.\\
Let $F=\bigcup_{k=1}^{n}(t_k,s_k)\subset T\times T$. We observe that, since $\{t_k\}_{k=1,...,n}$ are pairwise incomparable elements, we have $1_F\in \Y_T$. Therefore we obtain
\begin{equation*}
1_F\in \bigcap_{(t,s)\in A}V_{(t,s)}\cap\bigcap_{i=1}^{n} (\Y_T\setminus \bigcap_{(t,s)\in B_j}V_{(t,s)}).
\end{equation*}
Hence $\I$ satisfies also $(ii)$, so $\Y_T$ is not bisequential at $1_O$.
\end{proof}

Combining the above observations with Proposition \ref{biseqSuslin} we get the following characterization of $\Y_T$ spaces in terms of bisequentiality.

\begin{thm}
Let $T$ be a tree. The topological space $\Y_T$ is bisequential if and only if $T$ is countable.
\end{thm}

We recall that each uniform Eberlein compact space of weight less than the first measurable cardinal is bisequential \cite[Theorem 5.2]{Ciesla} (the original result is due to P. Nyikos), thus we can state, using the previous theorem, the following characterization of $\Y_T$ uniform Eberlein spaces.

\begin{cor}
Let $T$ be a tree. The space $\Y_T$ is uniform Eberlein if and only if is metrizable if and only if $T$ is countable.
\end{cor}


\begin{thebibliography}{999}

\addcontentsline{toc}{chapter}{Bibliography}
\bibitem{BenStar} Y. Benyamini, T. Starbird.,
{\em Embedding weakly compact sets into Hilbert space,} Israel J. Math. \textbf{23} (1976), 137-141.
\bibitem{Ciesla} T. Cie\'{s}la,
{\em A filter on a collection of finite sets and Eberlein compacta,} Topol. Appl. 
\textbf{229} (2017), 106-111.
\bibitem{Jech} T. Jech,
{\em Set theory,} Springer Monograph in Mathematics. Springer, Berlin, (2003) (The third edition, revised and expanded).
\bibitem{LeidSoko} A. G. Leiderman, G. A. Sokolov,
{\em Adequate families of sets and Corson compacts,} Comment. Math. Univ. Carolin. \textbf{25} (2) (1984), 233-246.
\bibitem{Michael} E.A. Michael,
{\em A quintuple quotient quest,} Gen. Topol. Appl. \textbf{2} (1972), 91-138.
\bibitem{Negre} S. Negrepontis,
{\em Banach spaces and topology,} In: K. Kunen, J.E. Vaughan (Eds.), Handbook of Set-Theoretic Topology, North-Holland, Amsterdam, 1984, 1045-1142.
\bibitem{Nyikos3} P. Nyikos,
{\em Properties of Eberlein compacta,} in: Abstract of English Summer Conference on General Topology and Applications, 1992, p.28.
\bibitem{Tal} M. Talagrand
{\em Espaces de Banach faiblement $K$-analitiques,} Ann. of Math. \textbf{110} (1979), 407-438.

\end{thebibliography}
\end{document}